\documentclass[12pt,reqno]{amsart}

\usepackage{amsmath,amssymb,amsthm,mathtools}
\usepackage{enumitem}
\usepackage[
colorlinks=true,
allcolors=blue,
hypertexnames=false
]{hyperref}
\usepackage{microtype}
\usepackage{comment}
\hypersetup{
  pdftitle={The Discrete Lp Minkowski Problem  for Negative p},
  pdfauthor={Junjie Shan}
}

\numberwithin{equation}{section}

\newtheorem{theorem}{Theorem}[section]
\newtheorem{proposition}[theorem]{Proposition}
\newtheorem{lemma}[theorem]{Lemma}
\newtheorem{corollary}[theorem]{Corollary}

\theoremstyle{remark}
\newtheorem{remark}[theorem]{Remark}

\DeclareMathOperator{\pos}{pos}
\DeclareMathOperator{\spanop}{span}
\DeclareMathOperator{\supp}{supp}

\newcommand{\R}{\mathbb R}
\newcommand{\Kn}{\mathcal K^n}
\newcommand{\Koo}{\mathcal K_o^n}
\newcommand{\Kee}{\mathcal K_e^n}
\newcommand{\Haus}{\mathcal H}
\newcommand{\Ncal}{\mathcal N}
\newcommand{\Dcal}{\mathcal D}

\title[The Discrete $L_p$ Minkowski Problem  for Negative $p$]{The Discrete $L_p$ Minkowski Problem  for Negative $p$}
\author{Junjie Shan}
\address{
	School of Mathematics,
	Sichuan University,
	Chengdu 610064,
	Sichuan,
	China
}
\email{shanjjmath@163.com}
\date{}


\begin{document}

\begin{abstract}

	In this paper, we establish a homological existence criterion for
	the discrete $L_p$ Minkowski problem for all $p<0$, where $p=-n$
	corresponds to the celebrated centro-affine case. For discrete measures with antipodal support, we obtain a complete
	existence characterization for arbitrary positive masses. In the even case, a solution can be chosen
	origin-symmetric.

\end{abstract}
\maketitle

\section{Introduction}

A central problem in convex geometry is to characterize geometric measures generated by convex bodies. The classical Minkowski problem asks for necessary and sufficient conditions under which a finite Borel measure on the unit sphere is the surface area measure of a convex body. The surface area measure may be viewed as the differential of the volume functional, and the solution of the Minkowski problem is a cornerstone of the Brunn--Minkowski theory with substantial influence on geometric analysis and fully nonlinear partial differential equations. 

Lutwak \cite{Lutwak} introduced the $L_p$ surface area measure
and posed the associated $L_p$ Minkowski problem.
This problem extends the classical Minkowski problem, which
corresponds to $p=1$, and has been  solved for $p\ge1$
\cite{ChouWang,HugLutwakYangZhang,Lutwak,LutwakYangZhang}.
For $0<p<1$, general existence results were established in
\cite{BianchiBoroczkyColesantiYang,ChenLiZhuMongeAmpere}.
The case $p=0$ is the celebrated logarithmic Minkowski problem,
which characterizes cone-volume measures and is closely related
to the logarithmic Brunn--Minkowski inequality
\cite{BoroczkyLutwakYangZhangLogBM}.
B\"or\"oczky, Lutwak, Yang, and Zhang
\cite{BoroczkyLutwakYangZhangLogM} completely solved its existence
problem for even measures in the class of origin-symmetric
convex bodies; for progress in the non-even case, see
\cite{BoroczkyHegedusZhu,ChenLiZhuLog,StancuDiscrete,ZhuLog}.

The $L_p$ Minkowski problem remains largely open for $p<0$.
Of particular importance is the critical case $p=-n$, known as
the centro-affine Minkowski problem. In the smooth setting,
it amounts to prescribing the affine distance, a longstanding
open problem in affine differential geometry.

For $-n<p<0$, Chou and Wang \cite{ChouWang} and
Bianchi, B\"or\"oczky, Colesanti, and Yang
\cite{BianchiBoroczkyColesantiYang} established existence
for broad classes of absolutely continuous measures.
In the plane, Chen \cite{ChenPlanar} established the corresponding
existence result for continuous even densities.
Bryan, Ivaki, and Scheuer \cite{BryanIvakiScheuer} gave a
unified curvature-flow approach to the smooth even problem
for $p>-n$.
For $p<-n$, Guang, Li, and Wang \cite{GuangLiWang}
proved existence for positive smooth densities.
In the critical case $p=-n$, Jian, Lu, and Zhu
\cite{JianLuZhu} established sufficient conditions for
existence within the class of convex bodies symmetric
with respect to the coordinate hyperplanes;
other partial results were obtained in
\cite{GuangLiWangCentro,JianLuWang,LiCentroAffineMultiplicity,LuWang}.
For discrete data, Zhu \cite{ZhuCentroAffine,ZhuNegative} proved existence for
every $p<0$ when the support is not contained in a closed
hemisphere and has no essential subspace.
The latter condition excludes supports containing an antipodal
pair; in particular, the result does not apply to  even
 measures.
For further results on the $L_p$ Minkowski problem with $p<0$,
see \cite{BianchiBoroczkyColesantiSmooth,
	DouZhu,DuPlanar,JianLuWangNonuniqueness,MilmanIsospectral}.

In the author's earlier work \cite{ShanGroup}, existence was
established for all $p<0$ when the prescribed measure is invariant
under an irreducible subgroup of $O(n)$.

Many new geometric measures of convex bodies have been introduced
and extensively studied. The associated characterization problems
include the Orlicz Minkowski problem, the dual Minkowski problem,
the $L_p$ Aleksandrov problem, the chord Minkowski problem,
and the affine dual Minkowski problem; see
\cite{CaiLengWuXiAffineDual,HaberlLutwakYangZhang,HuangLutwakYangZhangDual,%
	HuangLutwakYangZhangAleksandrov,HuangYangZhangSurvey,LutwakXiYangZhangChord}.

In this paper, we study the discrete $L_p$ Minkowski problem for
all $p<0$. We establish a homological sufficient condition for
existence that depends only on the prescribed normal directions.
This yields a complete existence characterization for discrete
measures with antipodal support. When the measure is even, a solution can be chosen
origin-symmetric.

The proof combines a variational argument with algebraic topology.
We associate a simplicial complex with the prescribed directions.
Nonvanishing reduced homology of this complex ensures an interior
critical point of the constrained volume functional and hence
a solution.

For a finite set $U=\{u_1,\ldots,u_N\}\subset S^{n-1}$, define its positive hull by
\begin{equation}\label{eq:intro-positive-hull}
	\pos U
	=\left\{\sum_{i=1}^N\lambda_i u_i:\lambda_i\geq0\right\}.
\end{equation}
Assume that $\pos U=\R^n$, set $E=\{1,\ldots,N\}$, and define the positive-spanning deletion complex
\begin{equation}\label{eq:intro-deletion-complex}
	\Dcal_+(U)
	=\left\{J\subset E:
	\pos\{u_i:i\in E\setminus J\}=\R^n\right\}.
\end{equation}
Thus, a set of indices belongs to $\Dcal_+(U)$ precisely when the directions remaining after those indices are deleted still positively span $\R^n$. We write $\widetilde H_k$ for
the $k$th reduced homology group, with the convention
$\widetilde H_{-1}(\{\varnothing\};\mathbb Z_2)=\mathbb Z_2$.

\begin{theorem}\label{thm:intro-homological}
	Let $p<0$, let $U=\{u_1,\ldots,u_N\}\subset S^{n-1}$ consist of distinct points, and suppose that $\pos U=\R^n$. If
	\begin{equation}\label{eq:intro-homology-condition}
		\widetilde H_k\bigl(\Dcal_+(U);\mathbb Z_2\bigr)\neq0
	\end{equation}
	for some $k\geq-1$, then, for every $\alpha_1,\ldots,\alpha_N>0$, there exists a polytope $P$ containing the origin in its
	interior such that
	\[
	S_p(P,\cdot)=\sum_{i=1}^N\alpha_i\delta_{u_i}.
	\]

\end{theorem}

The homological condition in Theorem~\ref{thm:intro-homological} depends only on $U$ and guarantees solvability for every choice
of positive masses. If $U$ is positively spanning and has no essential subspace, then the required reduced homology is nonzero, hence Theorem~\ref{thm:intro-homological} contains the result in
\cite{ZhuNegative} as a special case.

For antipodal supports spanning $\R^n$, the homological condition
in Theorem~\ref{thm:intro-homological} always holds.
This yields the following complete existence characterization,
without requiring the masses at opposite points to agree.

\begin{theorem}\label{thm:intro-antipodal}
	Let \(p<0\), let \(u_1,\ldots,u_N\in S^{n-1}\) satisfy
	\(u_i\neq\pm u_j\) whenever \(i\neq j\), and let
	\(\alpha_i^+,\alpha_i^->0\) for \(1\leq i\leq N\). Set
	\[
	\mu
	=
	\sum_{i=1}^N
	\bigl(\alpha_i^+\delta_{u_i}+\alpha_i^-\delta_{-u_i}\bigr).
	\]
	There exists a polytope \(K\) containing the origin in its
	interior such that
	\begin{equation}\label{eq:intro-antipodal-target}
		S_p(K,\cdot)=\mu
	\end{equation}
	if and only if
	\begin{equation}\label{eq:intro-spanning-condition}
		\spanop\{u_1,\ldots,u_N\}=\R^n.
	\end{equation}
\end{theorem}

For even measures, restricting the variational argument to
origin-symmetric polytopes yields the following characterization.

\begin{corollary}\label{cor:intro-even}
	Let \(p<0\), and let \(\mu\) be a nonzero finite even discrete
	Borel measure on \(S^{n-1}\). There exists an origin-symmetric $n$-dimensional polytope
	\(K\) such that
	\begin{equation}\label{eq:intro-even-target}
		S_p(K,\cdot)=\mu
	\end{equation}
	if and only if \(\mu\) is not concentrated on any great subsphere.
\end{corollary}

We also give a planar example showing that, for general discrete
measures, positive spanning of the support alone does not guarantee
solvability of the $L_p$ Minkowski problem. For this fixed support,
the homological condition in Theorem~\ref{thm:intro-homological}
fails, and solvability depends on the choice of positive masses.

\section{Preliminaries}\label{sec:preliminaries}

 Throughout, $n\geq2$.
The Euclidean inner product and norm on $\R^n$ are denoted by
$\langle\cdot,\cdot\rangle$ and $|\cdot|$.
The unit ball and unit sphere are denoted by $B$ and $S^{n-1}$.
We write $V$ for $n$-dimensional Lebesgue measure, $\Haus^k$
for $k$-dimensional Hausdorff measure, and $\sigma$ for spherical
Lebesgue measure on $S^{n-1}$. For background on convex bodies, see the books of
Schneider \cite{Schneider} and Gardner \cite{Gaedner gtbook}.

Let $\Kn$ denote the class of convex bodies in $\R^n$,
let $\Koo$ denote the class of convex bodies containing the
origin in their interiors, and let $\Kee$ denote the class of
origin-symmetric convex bodies.
A discrete measure means a positive Borel measure with finite
support.

For a compact convex set $K\subset\R^n$, its support function is
\begin{equation}
	h_K(u)=\max\{\langle x,u\rangle:x\in K\},
	\qquad u\in S^{n-1}.
\end{equation}
For $K\in\Koo$, let $S_K$ denote its surface area measure. More precisely, let $\nu_K$ denote the outer unit normal map,
which is uniquely defined at $\Haus^{n-1}$-almost every point
of $\partial K$. For every Borel set $\omega\subset S^{n-1}$,
\begin{equation}\label{eq:surface-area-measure}
	S_K(\omega)
	=
	\Haus^{n-1}\bigl(\nu_K^{-1}(\omega)\bigr).
\end{equation}
For $p\in\R$, the $L_p$ surface area measure of $K$ is defined by
\begin{equation}\label{eq:intro-Lp-measure}
	dS_p(K,\cdot)=h_K^{1-p}\,dS_K.
\end{equation}
If $P\in\Koo$ is a polytope, $u$ is an outer unit normal of a facet of $P$, and that facet has area $a_P(u)$, then \eqref{eq:intro-Lp-measure} gives
\begin{equation}\label{eq:polytope-Lp-mass}
 S_p(P,\{u\})=h_P(u)^{1-p}a_P(u).
\end{equation}
For $c>0$, the scaling relation is
\begin{equation}\label{eq:Lp-scaling}
 S_p(cK,\cdot)=c^{n-p}S_p(K,\cdot).
\end{equation}

For a finite set $U\subset S^{n-1}$, its positive hull $\pos U$ is defined by \eqref{eq:intro-positive-hull}. The polar cone of $\pos U$ is
\[
 (\pos U)^\circ
 =\{v\in\R^n:\langle v,u\rangle\leq0\text{ for every }u\in U\}.
\]
Consequently,
\begin{equation}\label{eq:positive-span-separation}
 \begin{aligned}
  \pos U=\R^n
  &\quad\Longleftrightarrow\quad (\pos U)^\circ=\{0\}\\
  &\quad\Longleftrightarrow\quad
  U\text{ is not contained in a closed hemisphere}.
 \end{aligned}
\end{equation}
If $U=-U$, then $\pos U=\spanop U$.

We shall  use the following elementary consequence of the classical Minkowski theory.

\begin{lemma}\label{lem:discrete-solution-polytope}
Let $p\in\R$, let $\mu$ be a finite discrete Borel measure on $S^{n-1}$, and suppose that $K\in\Koo$ satisfies $S_p(K,\cdot)=\mu$. Then $K$ is a polytope,
\begin{equation}\label{eq:support-equality}
 \supp S_K=\supp\mu,
\end{equation}
and $\pos(\supp\mu)=\R^n$.
\end{lemma}

\begin{proof}
Since $h_K$ is strictly positive on $S^{n-1}$, equation \eqref{eq:intro-Lp-measure} gives $dS_K=h_K^{p-1}\,d\mu$. Thus, $S_K$ is a finite discrete measure and \eqref{eq:support-equality} holds. Since $S_K$ is the surface area measure of a convex body with nonempty interior, its centroid is at the origin and it is not concentrated on any great subsphere. By the discrete Minkowski existence theorem \cite[Theorem~8.2.1]{Schneider}, there is a polytope $P$ such that $S_P=S_K$. Minkowski's uniqueness theorem \cite[Theorem~8.1.1]{Schneider} then implies that $K$ is a translate of $P$. Hence $K$ is a polytope. The facet normals of a bounded full-dimensional polytope positively span $\R^n$, proving the last assertion.
\end{proof}

Let \(E\) be a finite vertex set, let
\(\Delta_E=\{J:J\subset E\}\) be the simplex on \(E\), and let
\(\partial\Delta_E=\{J:J\subsetneq E\}\) be its boundary complex.
For each simplicial subcomplex \(\mathcal C\subset\Delta_E\), we write
\(\lvert\mathcal C\rvert\) for its geometric realization.
For example, if \(E=\{1,2,3\}\), then \(\lvert\Delta_E\rvert\) is a
triangle together with its interior, whereas
\(\lvert\partial\Delta_E\rvert\) is the union of its three edges.
Let \(\mathcal A\) be a simplicial subcomplex of
\(\partial\Delta_E\). Its combinatorial Alexander dual is
\begin{equation}\label{eq:combinatorial-Alexander-dual}
 \mathcal A^*=\{J\subset E:E\setminus J\notin\mathcal A\}.
\end{equation}
The standard complement model for combinatorial Alexander duality
\cite[Introduction]{BjornerTancer} shows that
\begin{equation}\label{eq:Alexander-complement-model}
	\lvert\partial\Delta_E\rvert
	\setminus
	\lvert\mathcal A\rvert
	\simeq
	\lvert\mathcal A^*\rvert.
\end{equation}
 If \(E\) has \(m\) elements, combinatorial Alexander duality \cite[Theorem~1.1]{BjornerTancer} gives
\begin{equation}
 \widetilde H_j(\mathcal A^*;\mathbb Z_2)
 \cong
 \widetilde H^{m-j-3}(\mathcal A;\mathbb Z_2).
\end{equation}
Since $\mathbb Z_2$ is a field, the universal coefficient theorem implies
\begin{equation}\label{eq:Alexander-nonvanishing}
 \widetilde H_j(\mathcal A^*;\mathbb Z_2)\neq0
 \quad\Longleftrightarrow\quad
 \widetilde H_{m-j-3}(\mathcal A;\mathbb Z_2)\neq0.
\end{equation}
We use reduced homology with the convention
\begin{equation}\label{eq:minus-one-homology}
 \widetilde H_{-1}(\{\varnothing\};\mathbb Z_2)=\mathbb Z_2,
\end{equation}
where the realization of the complex $\{\varnothing\}$ is empty.

Let \(Z=(z_1,\ldots,z_m)\) be a finite list of points in \(S^{n-1}\),
where repetitions are allowed. We call \(Z\) \textbf{ample} if every open
hemisphere contains at least one of the points \(z_i\), or equivalently, if
the points \(z_1,\ldots,z_m\) are not contained in any closed
hemisphere. The
open-hemisphere complex of \(Z\) is the simplicial complex on the
vertex set \(\{1,\ldots,m\}\) defined by
\begin{equation}
 \mathcal O(Z)
 =\{I\subset\{1,\ldots,m\}:\{z_i:i\in I\}\text{ is contained  in an open hemisphere}\}.
\end{equation}
In the notation of \cite{AyzenbergBeketovMagai}, this is the
constellation complex \(\operatorname{Stel}(Z)\).
If \(Z\) is ample, the open hemispheres centered at the points of
\(Z\) form a good cover of \(S^{n-1}\). In this case,
\cite[Theorem~1]{AyzenbergBeketovMagai} shows that the geometric
realization of \(\mathcal O(Z)\) is homotopy equivalent to
\(S^{n-1}\), that is,
\begin{equation}\label{eq:open-hemisphere-homotopy}
	\lvert\mathcal O(Z)\rvert\simeq S^{n-1}.
\end{equation}

For a finite configuration $U=\{u_1,\ldots,u_N\}\subset S^{n-1}$, define its non-positive-spanning complex by
\begin{equation}\label{eq:non-positive-spanning-complex}
 \Ncal_+(U)
 =\left\{I\subset\{1,\ldots,N\}:\pos\{u_i:i\in I\}\neq\R^n\right\}.
\end{equation}
A nonzero proper subspace $L\subset\R^n$ is called essential for $U$ if $U\cap L$ is not contained in any closed hemisphere of the relative sphere $L\cap S^{n-1}$.

\begin{lemma}\label{lem:no-essential-subspaces}
	Let \(U=\{u_1,\ldots,u_N\}\subset S^{n-1}\) be a finite set that is
	not contained in any closed hemisphere and has no essential subspaces.
	Then
	\begin{equation}\label{eq:open-equals-closed}
		\Ncal_+(U)=\mathcal O(U).
	\end{equation}
	Consequently,
	\begin{equation}\label{eq:zhu-deletion-homology}
		\widetilde H_{N-n-2}\bigl(\Dcal_+(U);\mathbb Z_2\bigr)\neq0.
	\end{equation}
\end{lemma}

\begin{proof}
Every subset contained in an open hemisphere fails to positively span $\R^n$, so $\mathcal O(U)\subset\Ncal_+(U)$. Suppose that $I\in\Ncal_+(U)$ does not belong to $\mathcal O(U)$. By separation, there is a nonzero vector $v$ such that $\langle v,u_i\rangle\geq0$ for all $i\in I$. On the other hand, failure to lie in an open hemisphere is equivalent, by strict separation, to
\[
 0\in\operatorname{conv}\{u_i:i\in I\}.
\]
Choose an inclusion-minimal subset \(C\subset I\) such that
\(0\in\operatorname{conv}\{u_i:i\in C\}\). Then there exist coefficients
\(\lambda_i>0\), \(i\in C\), such that
\(\sum_{i\in C}\lambda_i=1\) and
\(\sum_{i\in C}\lambda_i u_i=0\). Taking the inner product of \(\sum_{i\in C}\lambda_i u_i=0\) with
\(v\),  we obtain
\(\langle v,u_i\rangle=0\) for every \(i\in C\). Hence $L=\spanop\{u_i:i\in C\}$ is a nonzero proper subspace.

The set $\{u_i:i\in C\}$ is not contained in a closed hemisphere of $L\cap S^{n-1}$. Indeed, if $w\in L$ satisfied $\langle w,u_i\rangle\geq0$ for all $i\in C$, then  taking the
inner product of \(\sum_{i\in C}\lambda_i u_i=0\) with \(w\) would give $\langle w,u_i\rangle=0$ for every $i\in C$. Since these vectors span $L$, this would force $w=0$. Thus, $L$ is essential for $U$, a contradiction. This proves \eqref{eq:open-equals-closed}.

The hypothesis that $U$ is not contained in any closed hemisphere implies that $U$ is ample. Equations \eqref{eq:open-equals-closed} and
\eqref{eq:open-hemisphere-homotopy} show that
\(\lvert\Ncal_+(U)\rvert\) is homotopy equivalent to \(S^{n-1}\).
Hence
\(\widetilde H_{n-1}(\Ncal_+(U);\mathbb Z_2)\cong\mathbb Z_2\).  By the definition of \(\Dcal_+(U)\) in
\eqref{eq:intro-deletion-complex}, together with
\eqref{eq:combinatorial-Alexander-dual}, we have
\(\Dcal_+(U)=\Ncal_+(U)^*\). Applying
\eqref{eq:Alexander-nonvanishing} with
\(\mathcal A=\Ncal_+(U)\), \(m=N\), and \(j=N-n-2\) yields
\eqref{eq:zhu-deletion-homology}. 
\end{proof}

The finite list \(Z=(z_1,\ldots,z_m)\) is called
\textbf{antipodal} if every point of \(S^{n-1}\) and its antipode
occur equally many times in \(Z\). By
\eqref{eq:positive-span-separation}, \(I\in\Ncal_+(Z)\) if and only if
the points \(z_i\), \(i\in I\), are contained in a closed hemisphere.
In the terminology of \cite{AyzenbergBeketovMagai},
\(\Ncal_+(Z)=\operatorname{BStel}(Z)\).
 We shall use the following result
\textup{(\cite[Theorem~2 and Corollary~2.10]{AyzenbergBeketovMagai})}.

\begin{lemma}\label{thm:closed-hemisphere-topology}
	Let \(Z=(z_1,\ldots,z_m)\) be an ample antipodal finite list of points
	in \(S^{n-1}\). Then 
	\(\lvert\Ncal_+(Z)\rvert\) is homotopy equivalent to a nonempty wedge
	of \((2n-2)\)-dimensional spheres.
\end{lemma}

For the even case, we shall use a deletion complex defined by
linear spanning rather than positive spanning. The following
lemma establishes the nonvanishing of its reduced homology.

\begin{lemma}\label{lem:even-deletion-homology}
	Let \(U=\{u_1,\ldots,u_N\}\subset S^{n-1}\) span \(\R^n\), and let
	\(E=\{1,\ldots,N\}\). Define
	\begin{equation}\label{eq:even-deletion-complex}
		\Dcal_e(U)
		=
		\left\{
		J\subset E:
		\spanop\{u_i:i\in E\setminus J\}=\R^n
		\right\}.
	\end{equation}
	Then
	\begin{equation}\label{eq:even-deletion-homology}
		\widetilde H_{N-n-1}
		\bigl(\Dcal_e(U);\mathbb Z_2\bigr)\neq0.
	\end{equation}
\end{lemma}

\begin{proof}
	Since \(U\) spans \(\R^n\), the empty set belongs to
	\(\Dcal_e(U)\). If \(J\in\Dcal_e(U)\) and \(J'\subset J\), then
	\(E\setminus J'\) contains \(E\setminus J\); hence the vectors
	\(u_i\), \(i\in E\setminus J'\), still span \(\R^n\). Thus
	\(\Dcal_e(U)\) is a simplicial complex.
	
	Define
	\[
	\mathcal I(U)
	=
	\left\{
	I\subset E:
	\{u_i:i\in I\}\text{ is linearly independent in }\R^n
	\right\},
	\]
	and let \(M(U)=(E,\mathcal I(U))\) be the vector matroid represented
	by \(u_1,\ldots,u_N\). A basis of \(M(U)\) is a set
	\(B\in\mathcal I(U)\) that is maximal with respect to inclusion.
Since
\(u_1,\ldots,u_N\) span \(\R^n\), such a set \(B\) is precisely an index
set for a vector basis of \(\R^n\).  Hence every basis of \(M(U)\) has \(n\)
	elements.
	
	The bases of the dual matroid \(M(U)^*\) are, by definition, the sets
	\(E\setminus B\), where \(B\) ranges over the bases of \(M(U)\). A
	subset \(J\subset E\) is independent in \(M(U)^*\) precisely when
	\(J\subset E\setminus B\) for some basis \(B\) of \(M(U)\). Following
	\cite{BjornerMatroids}, write
	\[
	\operatorname{IN}\bigl(M(U)^*\bigr)
	=
	\left\{
	J\subset E:
	J\subset E\setminus B
	\text{ for some basis \(B\) of \(M(U)\)}
	\right\}.
	\]
	Thus \(M(U)^*=(E,\operatorname{IN}(M(U)^*))\).
	
	For \(J\subset E\), the condition
	\(J\in\operatorname{IN}(M(U)^*)\) is equivalent to
	\(E\setminus J\) containing a basis of \(M(U)\). Since the bases of
	\(M(U)\) are precisely the index sets of vector bases of \(\R^n\),
	this is equivalent to \(\{u_i:i\in E\setminus J\}\) spanning
	\(\R^n\). Therefore,
	\begin{equation}\label{eq:even-deletion-independence-complex}
		\Dcal_e(U)=\operatorname{IN}\bigl(M(U)^*\bigr).
	\end{equation}
	
	Since \(U\) spans \(\R^n\), necessarily \(N\geq n\). Every basis of
	\(M(U)^*\) is of the form \(E\setminus B\), where \(B\) is a basis
	of \(M(U)\), and therefore has \(N-n\) elements. Hence \(M(U)^*\)
	has rank \(N-n\), where the rank of a matroid is the common
	cardinality of its bases.
	
	Assume first that \(N>n\). We verify that no element of \(E\) belongs
	to every basis of \(M(U)^*\). Fix \(i\in E\). Since \(u_i\neq0\) and
	\(U\) spans \(\R^n\), the vector \(u_i\) can be extended by vectors
	from \(U\) to a vector basis of \(\R^n\). Thus there is a basis
	\(B_i\) of \(M(U)\) with \(i\in B_i\). The set \(E\setminus B_i\)
	is a basis of \(M(U)^*\) that does not contain \(i\). Since \(i\) was
	arbitrary, no element of \(E\) belongs to every basis of \(M(U)^*\).
	
	By \cite[Theorem~7.8.1]{BjornerMatroids}, if \(Q\) is a rank-\(s\)
	matroid, then the reduced integral homology of
	\(\operatorname{IN}(Q)\) vanishes outside degree \(s-1\), and its
	homology group in degree \(s-1\) is free abelian. By
	\cite[Exercise~7.39(a)]{BjornerMatroids}, all reduced integral
	homology groups of \(\operatorname{IN}(Q)\) vanish if and only if
	some element belongs to every basis of \(Q\). Applying these facts
	to \(Q=M(U)^*\), and using
	\eqref{eq:even-deletion-independence-complex}, we conclude that
	\(\widetilde H_{N-n-1}(\Dcal_e(U);\mathbb Z)\) is a nonzero free
	abelian group. Hence, for some integer \(b\geq1\),
	\[
	\widetilde H_{N-n-1}
	\bigl(\Dcal_e(U);\mathbb Z\bigr)
	\cong \mathbb Z^b.
	\]
	Since
	\(\widetilde H_{N-n-2}(\Dcal_e(U);\mathbb Z)=0\), the universal
	coefficient theorem gives
	\[
	\widetilde H_{N-n-1}
	\bigl(\Dcal_e(U);\mathbb Z_2\bigr)
	\cong
	\widetilde H_{N-n-1}
	\bigl(\Dcal_e(U);\mathbb Z\bigr)
	\otimes_{\mathbb Z}\mathbb Z_2
	\cong
	(\mathbb Z_2)^b\neq0.
	\]
	
	If \(N=n\), then \(u_1,\ldots,u_N\) form a vector basis of \(\R^n\).
	Consequently, no proper subfamily spans \(\R^n\), and
	\(\Dcal_e(U)=\{\varnothing\}\). Since \(N-n-1=-1\),
	\eqref{eq:even-deletion-homology} follows from the convention
	\eqref{eq:minus-one-homology}.
\end{proof}

\section{The parameter space in \texorpdfstring{\(p\)-power}{p-power} coordinates}\label{sec:parameter}

\subsection{Volume with fixed normal directions}
\label{subsec:fixed-normal-volume}
Fix $p<0$. Throughout Sections~\ref{sec:parameter}--\ref{sec:topology}, let $U=\{u_1,\ldots,u_N\}\subset S^{n-1}$ consist of distinct points and satisfy $\pos U=\R^n$. For $h=(h_1,\ldots,h_N)\in(0,\infty)^N$, set
\begin{equation}\label{eq:one-sided-polytope}
 P(h)=\bigcap_{i=1}^N
 \{x\in\R^n:\langle x,u_i\rangle\leq h_i\},
 \qquad
 \mathcal V(h)=V(P(h)).
\end{equation}
Since \(\pos U=\R^n\), each \(P(h)\) is bounded and contains the
origin in its interior. For each $i$, let
\begin{equation}
 F_i(h)=P(h)\cap\{x:\langle x,u_i\rangle=h_i\},
 \qquad
 a_i(h)=\Haus^{n-1}(F_i(h)).
\end{equation}
If the $i$-th inequality in \eqref{eq:one-sided-polytope} does not define a facet of $P(h)$, then
$F_i(h)$ is either empty or has dimension at most $n-2$, and hence
$a_i(h)=0$.

\begin{lemma}\label{lem:one-sided-volume-C1}
	The function \(\mathcal V\) in
	\eqref{eq:one-sided-polytope} belongs to
	\(C^1((0,\infty)^N)\), and
	\begin{equation}\label{eq:one-sided-volume-derivative}
		\frac{\partial\mathcal V}{\partial h_i}(h)=a_i(h),
		\qquad 1\leq i\leq N.
	\end{equation}
\end{lemma}

\begin{proof}
	Lemma~3.2 of \cite{HugLutwakYangZhang} applies directly:
	\(\pos U=\R^n\) is equivalent to \(U\) not being contained in a
	closed hemisphere, and the surface area measure of \(P(h)\) satisfies
	\(S(P(h),\{u_i\})=a_i(h)\).
\end{proof}

Introduce the \(p\)-power coordinates
\begin{equation}\label{eq:inverse-support-variables}
 t_i=h_i^p,
 \qquad 1\leq i\leq N,
\end{equation}
so that $h_i=t_i^{1/p}$. Define
\begin{equation}
 W(t)=\mathcal V(t_1^{1/p},\ldots,t_N^{1/p})
\end{equation}
and
\begin{equation}\label{eq:one-sided-mass}
 m_i(t)=t_i^{(1-p)/p}
 a_i(t_1^{1/p},\ldots,t_N^{1/p}).
\end{equation}
By \eqref{eq:polytope-Lp-mass}, $m_i(t)$ is the $L_p$ surface area mass at $u_i$.

\begin{corollary}\label{cor:inverse-volume-derivative}
The function $W$ belongs to $C^1((0,\infty)^N)$ and satisfies
\begin{equation}\label{eq:inverse-volume-derivative}
 \frac{\partial W}{\partial t_i}(t)=\frac1p m_i(t).
\end{equation}
Moreover,
\begin{equation}\label{eq:inverse-volume-homogeneity}
 W(ct)=c^{n/p}W(t),
 \qquad c>0.
\end{equation}
\end{corollary}

\begin{proof}
By Lemma~\ref{lem:one-sided-volume-C1} and the chain rule,
\[
 \frac{\partial W}{\partial t_i}(t)
 =\frac1p t_i^{1/p-1}
 a_i(t_1^{1/p},\ldots,t_N^{1/p})
 =\frac1p m_i(t),
\]
which proves \eqref{eq:inverse-volume-derivative}. Replacing $t$ by $ct$ multiplies every support number by $c^{1/p}$ and therefore multiplies the volume by $c^{n/p}$, proving \eqref{eq:inverse-volume-homogeneity}.
\end{proof}

\subsection{The weighted parameter simplex}
\label{subsec:weighted-parameter-simplex}

Fix $\alpha_1,\ldots,\alpha_N>0$, and define the weighted simplex
\begin{equation}\label{eq:weighted-simplex}
 \Delta_\alpha
 =\left\{t\in[0,\infty)^N:\sum_{i=1}^N\alpha_i t_i=1\right\}.
\end{equation}
Here and below, the interior and boundary of \(\Delta_\alpha\) are
understood relative to its affine hull. For $t\in\Delta_\alpha$, let
\begin{equation}
 J(t)=\{i:t_i>0\}.
\end{equation}
For every $t\in\Delta_\alpha$, define
\begin{equation}\label{eq:boundary-polytope}
	P(t)=\bigcap_{i\in J(t)}
	\{x\in\R^n:\langle x,u_i\rangle\leq t_i^{1/p}\}.
\end{equation}
Set
\begin{equation}\label{eq:parameter-volume-functional}
	W(t)=
	\begin{cases}
		V(P(t)),&\pos\{u_i:i\in J(t)\}=\R^n,\\
		+\infty,&\text{otherwise}.
	\end{cases}
\end{equation}

 This definition has the following geometric interpretation.
It follows from \eqref{eq:positive-span-separation} that \(P(t)\) is
bounded if and only if
\(\pos\{u_i:i\in J(t)\}=\R^n\). If this condition fails, then
\(P(t)\) has nonempty interior and contains an unbounded ray, so its
volume is infinite. On \(\operatorname{int}\Delta_\alpha\), this agrees with the preceding
definition of \(W\).

For \(I\subset E\), set
\[
\Delta_I
=
\left\{
t\in\Delta_\alpha:
t_j=0\text{ for every }j\in E\setminus I
\right\},
\]
with \(\Delta_\varnothing=\varnothing\), and define
\begin{equation}\label{eq:AXS}
	A=\bigcup_{I\in\Ncal_+(U)}\Delta_I,
	\qquad
	X=\Delta_\alpha\setminus A,
	\qquad
	S=\partial\Delta_\alpha\setminus A,
\end{equation}
where \(\Ncal_+(U)\) is defined in
\eqref{eq:non-positive-spanning-complex}.
\(A\) collects the ``bad'' faces of \(\Delta_\alpha\).
A point \(t\in\Delta_\alpha\) belongs to \(A\) if and only if
\(J(t)\in\Ncal_+(U)\). Hence
\begin{equation}\label{X}
X
=
\left\{
t\in\Delta_\alpha:
\pos\{u_i:i\in J(t)\}=\R^n
\right\},
\end{equation}
so \(X\) is exactly the set on which \(W\) is finite. Since
\(\pos U=\R^n\), the full index set \(E\) does not belong to
\(\Ncal_+(U)\); therefore \(A\subset\partial\Delta_\alpha\) and
\(X=\operatorname{int}\Delta_\alpha\cup S\).

\begin{lemma}\label{lem:W-regularity}
	The function \(W:X\to(0,\infty)\) is continuous and is
	\(C^1\) on \(\operatorname{int}\Delta_\alpha\). Moreover, for every
	\(t^0\in S\), there exist a neighborhood \(\mathcal U\) of \(t^0\)
	in \(\operatorname{aff}\Delta_\alpha\) and a function
	\(\widetilde W\in C^1(\mathcal U)\) such that
	\(
	\widetilde W=W\)
	on \(\mathcal U\cap X
	\).
\end{lemma}

\begin{proof}
	Since
	\(\operatorname{int}\Delta_\alpha\subset(0,\infty)^N\), the
	\(C^1\) assertion in the interior follows from
	Corollary~\ref{cor:inverse-volume-derivative}.
	
	Fix \(t^0\in S\), and set \(J=J(t^0)\). Since \(t^0\in X\), the
	directions \(u_j\), \(j\in J\), positively span \(\R^n\). 
	Since \(t_j^0>0\) for every \(j\in J\), set
	\(
	\eta=\frac12\min_{j\in J}t_j^0>0.
	\)
	Choose a sufficiently small neighborhood \(\mathcal U\) of \(t^0\)
	in \(\operatorname{aff}\Delta_\alpha\) such that
	\begin{equation}\label{low bound}
	t_j\geq\eta
	\qquad
	\text{for every }j\in J\text{ and }t\in\mathcal U.
	\end{equation}
	For \(t\in\mathcal U\), define
	\[
	P_J(t)
	=
	\bigcap_{j\in J}
	\{x\in\R^n:\langle x,u_j\rangle\leq t_j^{1/p}\},
	\qquad
	\widetilde W(t)=V(P_J(t)).
	\]
	Since
	\(
	\pos\{u_j:j\in J\}=\R^n,
	\)
	Corollary~\ref{cor:inverse-volume-derivative}, with the directions
	\(u_j\), \(j\in J\), in place of \(u_1,\ldots,u_N\), shows that
	\[
	(s_j)_{j\in J}
	\longmapsto
	V\left(
	\bigcap_{j\in J}
	\{x:\langle x,u_j\rangle\leq s_j^{1/p}\}
	\right)
	\]
	is \(C^1\) whenever \(s_j>0\) for every \(j\in J\).
	Since \(t_j\geq\eta>0\) on \(\mathcal U\) for \(j\in J\), and
	\(\widetilde W\) depends only on these coordinates, it follows that
	\(\widetilde W\in C^1(\mathcal U)\).

We next show that the polytopes \(P_J(t)\), \(t\in\mathcal U\),
are uniformly bounded. Set
	\[
	c_J
	=
	\min_{v\in S^{n-1}}
	\max_{j\in J}\langle v,u_j\rangle.
	\]
	Since
	\(\pos\{u_j:j\in J\}=\R^n\), 
	\eqref{eq:positive-span-separation} implies that, for every
	\(v\in S^{n-1}\), there exists \(j\in J\) such that
	\(\langle v,u_j\rangle>0\). Hence the continuous function
	\(
	v\longmapsto \max_{j\in J}\langle v,u_j\rangle
	\)
	is strictly positive on  \(S^{n-1}\). Therefore
	\(
	c_J
	>0
	\).
	 If
	\(x\in P_J(t)\) and \(x\neq0\), then
	\[
	c_J|x|
	\leq
	\max_{j\in J}\langle x,u_j\rangle
	\leq
	\max_{j\in J}t_j^{1/p}
	\leq
	\eta^{1/p},
	\]
	where the last inequality uses \(p<0\) and \eqref{low bound}. Hence
	\[
	P_J(t)\subset RB,
	\qquad
	R=\frac{\eta^{1/p}}{c_J},
	\qquad
	t\in\mathcal U.
	\]
	
	For every \(i\notin J\), we have \(t_i^0=0\). Since
	\(t_i^{1/p}\to+\infty\) as \(t_i\to0+\), we may shrink
	\(\mathcal U\) once more so that, whenever
	\(t\in\mathcal U\cap X\) and \(t_i>0\),
	\[
	t_i^{1/p}>R.
	\]
	Hence \(RB\subset\{x:\langle x,u_i\rangle\leq t_i^{1/p}\}\).
	Since \(P_J(t)\subset RB\), every point of \(P_J(t)\) already satisfies
	the constraint \(\langle x,u_i\rangle\leq t_i^{1/p}\) whenever
	\(i\notin J\) and \(t_i>0\). If \(t_i=0\), this constraint is absent by
	the definition of \(P(t)\). For \(t\in\mathcal U\cap X\), by \eqref{low bound} we have \(J\subset J(t)\). Hence, by the definition of
	\(P(t)\),
	\[
	P(t)
	=
	P_J(t)\cap
	\bigcap_{i\in J(t)\setminus J}
	\{x:\langle x,u_i\rangle\leq t_i^{1/p}\}.
	\]
	For every \(i\in J(t)\setminus J\), the preceding estimate shows that
	each \(x\in P_J(t)\) satisfies
	\(\langle x,u_i\rangle\leq t_i^{1/p}\). Therefore
	\[
	P(t)=P_J(t)
	\qquad
	\text{for every }t\in\mathcal U\cap X.
	\]
	
Thus, \(\widetilde W=W\) on \(\mathcal U\cap X\). Hence
\(\widetilde W\) is the required local \(C^1\) extension, and \(W\)
is continuous at \(t^0\) relative to \(X\). Since \(t^0\in S\)
was arbitrary and \(W\) is \(C^1\) on
\(\operatorname{int}\Delta_\alpha\), the equality
\(X=\operatorname{int}\Delta_\alpha\cup S\) shows that \(W\) is
continuous on \(X\).

\end{proof}

\begin{lemma}\label{lem:W-lower-bound}
There is a constant $c_0>0$ such that
\begin{equation}\label{eq:W-lower-bound}
 W(t)\geq c_0
 \qquad\text{for every }t\in X.
\end{equation}
\end{lemma}

\begin{proof}
The relation \eqref{eq:weighted-simplex} gives $t_i\leq\alpha_i^{-1}$ whenever $t_i>0$. Since $1/p<0$, every finite support number in \eqref{eq:boundary-polytope} is at least $\alpha_i^{-1/p}$. If
\begin{equation}\label{eq:fixed-inner-radius}
 r=\min_{1\leq i\leq N}\alpha_i^{-1/p},
\end{equation}
then $rB\subset P(t)$ for every $t\in X$. Thus, \eqref{eq:W-lower-bound} holds with $c_0=V(rB)$.
\end{proof}

\begin{lemma}\label{lem:W-blowup}
If $t^k\in X$ and $t^k\to t^0\in A$, then
\begin{equation}\label{eq:W-blowup}
 W(t^k)\longrightarrow+\infty.
\end{equation}
\end{lemma}

\begin{proof}
	Let \(J=J(t^0)\). Since \(t^0\in A\), the directions
	\(u_j\), \(j\in J\), do not positively span \(\R^n\).
	By \eqref{eq:positive-span-separation}, there exists a unit vector
	\(v\in\R^n\) such that
	\(\langle v,u_j\rangle\leq0\) for every \(j\in J\).
	
	Set
	\(I_+=\{i\in E\setminus J:\langle v,u_i\rangle>0\}\).
	This set is nonempty. Indeed, otherwise
	\(\langle v,u_i\rangle\leq0\) for every \(i\in E\), contradicting
	\(\pos U=\R^n\). Moreover,
	\(I_+\cap J(t^k)\neq\varnothing\) for every \(k\). Otherwise
	\(\langle v,u_i\rangle\leq0\) for every \(i\in J(t^k)\), contrary
	to \(t^k\in X\).
	
	For \(i\notin J\), define
	\begin{equation}\label{eq:vanishing-supports}
		\rho_i^k=
		\begin{cases}
			(t_i^k)^{1/p},&t_i^k>0,\\
			+\infty,&t_i^k=0.
		\end{cases}
	\end{equation}
	Since \(t_i^k\to t_i^0=0\) as \(k\to\infty\) and \(p<0\), we have
	\(\rho_i^k\to+\infty\) as \(k\to\infty\) for every \(i\notin J\). Define
	\[
	R_k
	=
	\min_{i\in I_+}
	\frac{\rho_i^k}{\langle v,u_i\rangle}.
	\]
	Since \(I_+\cap J(t^k)\neq\varnothing\), at least one term in this
	minimum is finite, and hence \(R_k<\infty\). Since \(I_+\) is a
	fixed finite set and every quotient in the minimum tends to
	\(+\infty\), we have \(R_k\to+\infty\).
	
	We claim that \(R_kv\in P(t^k)\). Let
	\(\ell\in J(t^k)\). If \(\langle v,u_\ell\rangle\leq0\), then
	\(\langle R_kv,u_\ell\rangle\leq0<(t_\ell^k)^{1/p}\).
	If \(\langle v,u_\ell\rangle>0\), then
	\(\ell\notin J\), so \(\ell\in I_+\), and the definition of \(R_k\)
	gives
	\(\langle R_kv,u_\ell\rangle\leq
	\rho_\ell^k=(t_\ell^k)^{1/p}\).
	Thus \(R_kv\) satisfies every inequality defining \(P(t^k)\).
	
	The fixed ball \(rB\) from \eqref{eq:fixed-inner-radius} also lies
	in \(P(t^k)\). Since \(P(t^k)\) is convex, it contains the cone
	\[
	C_k
	=
	\operatorname{conv}
	\bigl((rB\cap v^\perp)\cup\{R_kv\}\bigr).
	\]
	Since \(v\) is a unit vector, this cone has height \(R_k\).
	Consequently,
	\[
	W(t^k)
	\geq V(C_k)
	=
	\frac{R_k}{n}\,
	\Haus^{n-1}(rB\cap v^\perp)
	\longrightarrow+\infty.
	\]
	This proves \eqref{eq:W-blowup}.
\end{proof}

\begin{lemma}\label{lem:exit-derivative}
	Let \(t^0\in S\), and let \(i\notin J(t^0)\). For
	\(0\leq s<\alpha_i^{-1}\), define \(t(s)=(t_1(s),\ldots,t_N(s))\in\Delta_\alpha\) by
	\begin{equation}\label{eq:exit-curve}
		t_i(s)=s,
		\qquad
		t_j(s)=(1-\alpha_i s)t_j^0
		\quad(j\neq i).
	\end{equation}
	Then \(t(s)\in X\) for \(0<s<\alpha_i^{-1}\). Moreover, for all
	sufficiently small \(s>0\),
	\begin{equation}\label{eq:exit-scaling}
		W(t(s))
		=
		(1-\alpha_i s)^{n/p}W(t^0),
	\end{equation}
	and therefore
	\begin{equation}\label{eq:exit-derivative}
		\left.\frac{d}{ds}\right|_{s=0+}W(t(s))
		=
		-\frac np\alpha_iW(t^0)>0.
	\end{equation}
\end{lemma}

\begin{proof}
	Set \(J=J(t^0)\). Since \(i\notin J\), we have \(t_i^0=0\), and
	hence \(\sum_{j\neq i}\alpha_jt_j^0=1\). It follows directly from
	\eqref{eq:exit-curve} that
	\(\sum_j\alpha_jt_j(s)=1\). Thus \(t(s)\in\Delta_\alpha\) whenever
	\(0\leq s<\alpha_i^{-1}\).
	
	For \(0<s<\alpha_i^{-1}\), we have
	\(J(t(s))=J\cup\{i\}\). Since \(t^0\in S\subset X\), the directions
	\(u_j\), \(j\in J\), positively span \(\R^n\). The directions
	indexed by \(J(t(s))\) therefore also positively span \(\R^n\), so
	\(t(s)\in X\) by \eqref{X}.
	
	Set \(\lambda_s=(1-\alpha_i s)^{1/p}\). For every \(j\in J\),
	\(
	t_j(s)^{1/p}
	=
	\lambda_s(t_j^0)^{1/p}.
	\)
	Consequently, the definition of \(P(t(s))\) gives
	\[
	P(t(s))
	=
	\lambda_sP(t^0)
	\cap
	\{x\in\R^n:\langle x,u_i\rangle\leq s^{1/p}\}.
	\]
	
	The polytope \(P(t^0)\) is bounded, so choose \(R_0>0\) such that
	\(P(t^0)\subset R_0B\). Since \(p<0\), we have
	\(s^{1/p}\to+\infty\) as \(s\to0+\), whereas
	\(\lambda_s\to1\). Hence, for all sufficiently small \(s>0\),
	\(s^{1/p}>\lambda_sR_0\). For every
	\(x\in\lambda_sP(t^0)\), it follows that
	\(\langle x,u_i\rangle\leq|x|\leq\lambda_sR_0<s^{1/p}\).
	Thus
	\[
	P(t(s))=\lambda_sP(t^0)
	\]
	for all sufficiently small \(s>0\).
	
	Taking volumes yields \eqref{eq:exit-scaling}. Differentiating
	\eqref{eq:exit-scaling} at \(s=0+\) gives
	\eqref{eq:exit-derivative}; the final inequality follows from
	\(p<0\), \(\alpha_i>0\), and \(W(t^0)>0\).
\end{proof}

\section{Topology and the general existence theorem}\label{sec:topology}
All interiors and boundaries of simplices are understood relative to
their affine hulls. The following lemma shows that nontrivial relative homology, together with suitable boundary behavior, forces the existence of an interior critical point.

\begin{lemma}\label{lem:relative-critical-point}
	Let \(\Delta\) be a simplex, let
	\(A\subset\partial\Delta\) be a union of faces of \(\Delta\), and set
	\[
	X=\Delta\setminus A,
	\qquad
	S=\partial\Delta\setminus A.
	\]
	Let \(L=\operatorname{aff}\Delta\), and let \(T=L-L\) be its
	translation space. For each facet \(G\) of \(\Delta\), fix an
	affine function \(\ell_G\) on \(L\) such that
	\[
	G=\{x\in\Delta:\ell_G(x)=0\},
	\qquad
	\Delta\subset\{x\in L:\ell_G(x)\geq0\}.
	\]
	A vector \(Y\in T\) is said to point strictly outside \(G\) if
	\[
	D\ell_G[Y]
	=
	\left.\frac{d}{dt}\right|_{t=0}\ell_G(x+tY)<0,
	\]
	where \(x\in L\); the value is independent of \(x\) since
	\(\ell_G\) is affine.

	Let \(F:X\to\R\) be continuous and bounded below. Assume that:
	\begin{enumerate}[label=\textup{(\roman*)},leftmargin=2.4em]
		\item
		\(F(x_j)\to+\infty\) whenever \(x_j\in X\) and
		\(x_j\to a\in A\);
		
		\item
		\(F\) is \(C^1\) on \(\operatorname{int}\Delta\), and for
		every \(x\in S\), there are an open neighborhood \(U_x\) of
		\(x\) in \(L\) and a function
		\(\widetilde F_x\in C^1(U_x)\) such that
		\(\widetilde F_x=F\) on \(U_x\cap X\).
		For \(x\in S\), define
		\(
		DF(x)=D\widetilde F_x(x)
		\).
		This definition is independent of the chosen local extension;
		
		\item
		for every \(x\in S\), there is a vector \(Y_x\in T\) that
		points strictly outside every facet of \(\Delta\) containing
		\(x\) and satisfies
		\(
		DF(x)[Y_x]<0
		\).
	\end{enumerate}
	A point \(x\in\operatorname{int}\Delta\) is called a critical
	point of \(F\) if \(DF(x)=0\) as a linear functional on \(T\).
	If
	\[
	H_q(X,S;\mathbb Z_2)\neq0
	\]
	for some \(q\geq0\), then \(F\) has a critical point in
	\(\operatorname{int}\Delta\).
\end{lemma}

\begin{proof}
	If \(\dim\Delta=0\), then \(T=\{0\}\), and the conclusion is
	immediate. Hence assume that \(\dim\Delta\geq1\).

	Suppose that \(F\) has no critical point in
	\(\operatorname{int}\Delta\). For each
	\(x\in\operatorname{int}\Delta\), choose \(V_x\in T\) such that
	\(DF(x)[V_x]<0\). By continuity, there are an open neighborhood
	\(U_x\) of \(x\), with
	\(\overline{U_x}\subset\operatorname{int}\Delta\), and a constant
	\(c_x>0\) such that
	\[
	DF(y)[V_x]\leq-c_x
	\qquad
	\text{for every }y\in U_x.
	\]
	
	For \(x\in S\), choose a local extension
	\(\widetilde F_x\) as in \textup{(ii)}, and set
	\(V_x=Y_x\), where \(Y_x\) is given by \textup{(iii)}. After shrinking \(U_x\) in \textup{(ii)},
	there is a constant \(c_x>0\) such that
	\[
	D\widetilde F_x(y)[V_x]\leq-c_x
	\qquad
	\text{for every }y\in U_x.
	\]
	On \(U_x\cap\operatorname{int}\Delta\), this becomes
	\(DF(y)[V_x]\leq-c_x\).
	
	Since \(\Delta\) has finitely many facets and every facet not
	containing \(x\) has positive distance from \(x\), we may shrink
	\(U_x\) further so that every facet meeting \(U_x\) contains
	\(x\). By \textup{(iii)}, \(V_x\) therefore points strictly outside every
	facet meeting \(U_x\).
	
Thus, for every \(x\in X=\operatorname{int}\Delta\cup S\), we have
chosen an open neighborhood \(U_x\) of \(x\). Let
	\(\Omega=\bigcup_{x\in X}U_x\), which is an open neighborhood of
	\(X\) in \(L\). Choose a locally finite refinement
	\(\{U_\nu\}\) of the cover \(\{U_x\}_{x\in X}\), and a smooth
	partition of unity \(\{\psi_\nu\}\) on \(\Omega\) such that
	\(\operatorname{supp}\psi_\nu\subset U_\nu\) for every \(\nu\). For each \(\nu\), choose \(x_\nu\in X\) such that
	\(U_\nu\subset U_{x_\nu}\), and set
	\(V_\nu:=V_{x_\nu}\) and \(c_\nu:=c_{x_\nu}\).
	Define
	\[
	Y(y)=\sum_\nu\psi_\nu(y)\frac{V_\nu}{c_\nu},
	\qquad y\in\Omega.
	\]
	The sum is locally finite, so \(Y\) is a smooth vector field on
	\(\Omega\). For every \(y\in\operatorname{int}\Delta\), the linearity
	of \(DF(y)\) in the direction variable gives
	\begin{equation}\label{eq:unit-descent}
		DF(y)[Y(y)]
		=
		\sum_\nu\psi_\nu(y)\frac{DF(y)[V_\nu]}{c_\nu}
		\leq
		-\sum_\nu\psi_\nu(y)
		=-1.
	\end{equation}

	Let \(y\in S\), and let \(G\) be a facet containing \(y\).
	An interior neighborhood has closure contained in
	\(\operatorname{int}\Delta\), so it cannot contribute to \(Y(y)\).
If \(\psi_\nu(y)>0\), then
\(y\in U_\nu\subset U_{x_\nu}\). Since \(y\in G\), the facet \(G\)
meets \(U_{x_\nu}\). By the choice of \(U_{x_\nu}\), this implies
\(x_\nu\in G\), and hence \textup{(iii)} gives
\(D\ell_G[V_\nu]<0\). Hence
	\begin{equation}\label{DLG}
	D\ell_G[Y(y)]
	=
	\sum_\nu
	\frac{\psi_\nu(y)}{c_\nu}
	D\ell_G[V_\nu]
	<0.
	\end{equation}
	Thus \(Y(y)\) points strictly outside every facet containing \(y\).
	
	Let \(\varphi\) be the maximal local flow on \(\Omega\) generated by
	\(Y\); thus \(\varphi(x,0)=x\), and
	\(\partial_t\varphi(x,t)=Y(\varphi(x,t))\) whenever the flow is
	defined. Since \(F\) is bounded below, set
	\[
	m=\inf_X F>-\infty.
	\]
	Fix \(x\in\operatorname{int}\Delta\), and define
	\[
	\tau_*(x)
	=
	\sup\bigl\{
	t>0:
	\varphi(x,s)\ \text{is defined and belongs to }
	\operatorname{int}\Delta
	\text{ for every }0\leq s<t
	\bigr\}.
	\]
	
	Consider the sublevel set
	\[
	K_x=\{y\in X:F(y)\leq F(x)\}.
	\]
	We first note that \(K_x\) is compact in \(L\). Indeed, let
	\(y_j\in K_x\). Since \(\Delta\) is compact, after passing to a
	subsequence we may assume that \(y_j\to y\in\Delta\). The limit
	cannot belong to \(A\), because condition \textup{(i)} would then
	give \(F(y_j)\to+\infty\), whereas \(F(y_j)\leq F(x)\). Hence
	\(y\in X\), and the continuity of \(F\) gives
	\(F(y)\leq F(x)\). Thus \(y\in K_x\).
	
	For \(0\leq t<\tau_*(x)\), the chain rule and
	\eqref{eq:unit-descent} give
	\[
	\frac{d}{dt}F(\varphi(x,t))
	=
	DF(\varphi(x,t))
	\left[
	\frac{\partial}{\partial t}\varphi(x,t)
	\right]=
	DF(\varphi(x,t))[Y(\varphi(x,t))]
	\leq-1.
	\]
	Consequently,
	\[
m\le	F(\varphi(x,t))\leq F(x)-t.
	\]
	In particular,
	\[
	t\leq F(x)-m,
	\]
	and therefore
	\[
	\tau_*(x)\leq F(x)-m<\infty.
	\]
	Moreover, \(\varphi(x,t)\in K_x\) for every
	\(0\leq t<\tau_*(x)\).
	
	Since 
	\(K_x\) is a compact subset of \(\Omega\), the standard continuation
	theorem for ordinary differential equations implies that the solution
	\(t\mapsto\varphi(x,t)\) extends beyond \(t=\tau_*(x)\). In particular,
	\(\varphi(x,\tau_*(x))\) is well defined.  The continuity of the flow and the closedness
	of \(\Delta\) give
	\(\varphi(x,\tau_*(x))\in\Delta\).
	If this point belonged to \(\operatorname{int}\Delta\), then, since
	the solution extends beyond \(t=\tau_*(x)\) and
	\(\operatorname{int}\Delta\) is open, it would remain in
	\(\operatorname{int}\Delta\) for a short time after \(\tau_*(x)\).
	This contradicts the definition of \(\tau_*(x)\). Hence
	\[
	\varphi(x,\tau_*(x))\in\partial\Delta.
	\]

	Since \(\varphi(x,t)\in K_x\) for \(t<\tau_*(x)\) and \(K_x\) is
	closed, we also have
	\(
	\varphi(x,\tau_*(x))\in K_x\subset X
	\).
	It follows that
	\begin{equation}\label{taus}
	\varphi(x,\tau_*(x))\in
	\partial\Delta\cap X=S.
	\end{equation}
	If \(S=\varnothing\), this is already a contradiction.

Assume now that \(S\neq\varnothing\). For
\(x\in\operatorname{int}\Delta\), set
\(\tau(x)=\tau_*(x)\), which is the first time at which
\(\varphi(x,t)\) reaches \(\partial\Delta\), and set
\(\tau(x)=0\) for \(x\in S\). We prove that \(\tau\) is continuous.

Fix \(x_0\in\operatorname{int}\Delta\), and write
\(\tau_0=\tau(x_0)\). If \(0<\varepsilon<\tau_0\), then the set
\(\varphi(x_0,[0,\tau_0-\varepsilon])\) is a compact subset of
\(\operatorname{int}\Delta\), and hence has positive distance from
\(\partial\Delta\). By continuous dependence of the flow on the initial point
\cite[Theorem~6.1]{TeschlODE}, if \(x\) is sufficiently close to
\(x_0\), then \(\varphi(x,t)\in\operatorname{int}\Delta\) for
\(0\leq t\leq\tau_0-\varepsilon\). Hence
\[
\tau(x)\geq\tau_0-\varepsilon.
\]

	Set \(y_0=\varphi(x_0,\tau_0)\), and choose a facet \(G\)
	containing \(y_0\). Since \(y_0\in S\subset\Omega\), the flow is defined for a short
	time after \(\tau_0\). Since \(y_0\in G\), we have \(\ell_G(y_0)=0\).
	The flow equation
	\(\partial_t\varphi(x_0,t)=Y(\varphi(x_0,t))\)
	and the chain rule give
	\[
	\left.\frac{d}{dt}\ell_G(\varphi(x_0,t))\right|_{t=\tau_0}
	=
	D\ell_G[Y(y_0)]<0
	\]
	by \eqref{DLG}. Hence
	\[
	\ell_G(\varphi(x_0,\tau_0+\delta))<0
	\]
	for all sufficiently small \(\delta>0\).  Choose such a
	\(\delta\in(0,\varepsilon)\). By the continuous dependence of the flow on the initial
	point, if \(x\) is sufficiently close
	to \(x_0\), then \(\varphi(x,t)\) is defined for
	\(0\leq t\leq\tau_0+\delta\), and
	\[
	\ell_G(\varphi(x,\tau_0+\delta))<0.
	\]
Since \(\Delta\subset\{\ell_G\geq0\}\), the curve
\(t\mapsto\varphi(x,t)\) must meet \(\partial\Delta\) no later than
time \(\tau_0+\delta\). Hence
\[
\tau(x)\leq\tau_0+\delta<\tau_0+\varepsilon,
\]
for \(x\) sufficiently close to \(x_0\). Thus \(\tau\) is
continuous at every point of \(\operatorname{int}\Delta\).
	
	Finally, fix \(x_0\in S\), and choose a facet \(G\) containing
	\(x_0\). Since \(D\ell_G[Y(x_0)]<0\) by \eqref{DLG} and \(Y\) is smooth on
	\(\Omega\), there are bounded open
	neighborhoods \(U_0\) and \(U\) of \(x_0\), with
	\(\overline{U_0}\subset U\) and
	\(\overline U\subset\Omega\), and constants \(c,M>0\) such that
	\begin{equation}\label{bound}
	D\ell_G[Y(z)]\leq-c
	\qquad\text{and}\qquad
	|Y(z)|\leq M
	\quad\text{on }U.
	\end{equation}
	Set
	\[
	d=\operatorname{dist}(\overline{U_0},L\setminus U)>0.
	\]
	Let \(x\in X\cap U_0\) be sufficiently close to \(x_0\). If
	\(x\in S\), then \(\tau(x)=0\). Suppose that
	\(x\in\operatorname{int}\Delta\), and set
	\(T_x=\ell_G(x)/c\). Since \(\ell_G(x)\to\ell_G(x_0)=0\) as
	\(x\to x_0\), we may assume that \(MT_x<d\).
	
We claim that \(\tau(x)\leq T_x\). Suppose, to the contrary,
that \(\tau(x)>T_x\). Then
\(\varphi(x,t)\in\operatorname{int}\Delta\) for every
\(0\leq t\leq T_x\).
If the trajectory left \(U\) during this interval, let
\(\sigma\leq T_x\) be its first exit time from \(U\).
Since \(\partial_t\varphi(x,t)=Y(\varphi(x,t))\) and
\(\varphi(x,0)=x\), we would have
\[
d
\leq
|\varphi(x,\sigma)-x|
\leq
\int_0^\sigma |Y(\varphi(x,t))|\,dt
\leq
M\sigma
\leq
MT_x
<d,
\]
a contradiction.  Hence \(\varphi(x,t)\in U\) for every \(0\leq t\leq T_x\).
By \eqref{bound}, the chain rule gives
\(\frac{d}{dt}\ell_G(\varphi(x,t))
=D\ell_G[Y(\varphi(x,t))]\leq-c\) on \([0,T_x]\).
Integrating and using \(T_x=\ell_G(x)/c\), we obtain
\[
\ell_G(\varphi(x,T_x))
\leq
\ell_G(x)-cT_x
=0.
\]
	This contradicts \(\tau(x)>T_x\), because every interior point of
	\(\Delta\) satisfies \(\ell_G>0\). Therefore
	\[
	0\leq\tau(x)\leq\frac{\ell_G(x)}{c}\longrightarrow0
	\qquad
	\text{as }x\to x_0.
	\]
	Thus \(\tau\) is continuous on \(X\).
	
	Define
	\[
	H(x,s)=\varphi(x,s\tau(x)),
	\qquad
	x\in X,\quad 0\leq s\leq1.
	\]
	If \(x\in\operatorname{int}\Delta\), then
	\(H(x,s)\in\operatorname{int}\Delta\) for \(s<1\), while
	\(H(x,1)\in S\) by \eqref{taus}. If \(x\in S\), then \(\tau(x)=0\), so
	\(H(x,s)=x\) for every \(s\). Since \(\tau\) is continuous and
	\((x,t)\mapsto\varphi(x,t)\) is continuous on the domain of the
	local flow, \(H\) is continuous. Hence \(H\) is a strong
	deformation retraction of \(X\) onto \(S\).
	
	It follows that
	\[
	H_q(X,S;\mathbb Z_2)=0
	\qquad
	\text{for every }q\geq0,
	\]
	contrary to the hypothesis.
\end{proof}

\begin{lemma}\label{lem:relative-homology-general}
Let $X$ and $S$ be defined by \eqref{eq:AXS}. If
\begin{equation}\label{eq:deletion-homology-degree}
 \widetilde H_k(\Dcal_+(U);\mathbb Z_2)\neq0
\end{equation}
for some $k\geq-1$, then
\begin{equation}
 H_{k+1}(X,S;\mathbb Z_2)\neq0.
\end{equation}
\end{lemma}

\begin{proof}
	Choose \(b\in\operatorname{int}\Delta_\alpha\). For \(t\in X\) and \(s\in[0,1]\), the homotopy
	\begin{equation}\label{eq:X-contraction}
		H(t,s)=(1-s)t+sb
	\end{equation}
	contracts \(X\) to \(b\). Indeed, \(H(t,0)=t\) and \(H(t,1)=b\).
	Moreover, for \(s>0\), every coordinate of \(H(t,s)\) is positive.
	Since \(H(t,s)\in\Delta_\alpha\) and \(\pos U=\R^n\), it follows that
	\[
	H(t,s)\in\operatorname{int}\Delta_\alpha\subset X.
	\]
	Hence \(X\) is contractible.
	
		Let
	\[
	|\Delta_E|
	=
	\left\{
	\lambda=(\lambda_1,\ldots,\lambda_N)\in[0,\infty)^N:
	\sum_{i=1}^N\lambda_i=1
	\right\}
	\]
	be the standard simplex whose vertices are indexed by \(E\).
	The affine homeomorphism
	\[
	\Phi:\Delta_\alpha\longrightarrow\lvert\Delta_E\rvert,
	\qquad
	\Phi(t_1,\ldots,t_N)
	=
	(\alpha_1t_1,\ldots,\alpha_Nt_N),
	\]
		maps each face of \(\Delta_\alpha\) indexed by \(I\subset E\) onto
	the corresponding face of \(\lvert\Delta_E\rvert\). Since, by \eqref{eq:AXS},
	\[
	A=\bigcup_{I\in\Ncal_+(U)}\Delta_I,
	\qquad
	S=\partial\Delta_\alpha\setminus A,
	\]
 we have
	\[
	\begin{aligned}
		\Phi(A)
		&=
		\Phi\left(
		\bigcup_{I\in\Ncal_+(U)}\Delta_I
		\right)\\
		&=
		\bigcup_{I\in\Ncal_+(U)}\Phi(\Delta_I)\\
		&=
		\lvert\Ncal_+(U)\rvert.
	\end{aligned}
	\]
Recall that \(\lvert\Ncal_+(U)\rvert\) denotes the geometric
realization of \(\Ncal_+(U)\), namely the union of the faces indexed
by \(I\in\Ncal_+(U)\).	And
	\[
	\Phi(S)
	=
	\Phi(\partial\Delta_\alpha)\setminus\Phi(A)
	=
	\lvert\partial\Delta_E\rvert
	\setminus
	\lvert\Ncal_+(U)\rvert.
	\]

	 Since \(\Phi|_S\) is a homeomorphism from \(S\) onto
	\(\lvert\partial\Delta_E\rvert
	\setminus\lvert\Ncal_+(U)\rvert\), applying
	\eqref{eq:Alexander-complement-model} with
	\(\mathcal A=\Ncal_+(U)\), and then using
	\(\Ncal_+(U)^*=\Dcal_+(U)\), gives
	\[
	S
	\simeq
	\lvert\Ncal_+(U)^*\rvert
	=
	\lvert\Dcal_+(U)\rvert.
	\]
	
If \(k\geq0\), the reduced long exact sequence of the pair
\((X,S)\) contains
\[
\widetilde H_{k+1}(X;\mathbb Z_2)
\longrightarrow
H_{k+1}(X,S;\mathbb Z_2)
\longrightarrow
\widetilde H_k(S;\mathbb Z_2)
\longrightarrow
\widetilde H_k(X;\mathbb Z_2).
\]
Since \(X\) is contractible, its reduced homology groups vanish.
Hence
\[
H_{k+1}(X,S;\mathbb Z_2)
\cong
\widetilde H_k(S;\mathbb Z_2)\cong
\widetilde H_k(\Dcal_+(U);\mathbb Z_2).
\]
Therefore \eqref{eq:deletion-homology-degree} implies
\(H_{k+1}(X,S;\mathbb Z_2)\neq0\).

	If \(k=-1\), then \eqref{eq:deletion-homology-degree} and
	\eqref{eq:minus-one-homology} imply
	\(\Dcal_+(U)=\{\varnothing\}\). By the definition of
	\(\Dcal_+(U)\), every proper subset of \(E\) then belongs to
	\(\Ncal_+(U)\). Hence
	\(A=\partial\Delta_\alpha\) and \(S=\varnothing\). Since \(X\) is
	nonempty and contractible,
	\[
	H_0(X,S;\mathbb Z_2)
	=
	H_0(X;\mathbb Z_2)
	\cong
	\mathbb Z_2.
	\]
\end{proof}

\begin{proposition}\label{prop:interior-critical-point}
	Under the hypotheses of Theorem~\ref{thm:intro-homological}, the
	volume functional \(W\) defined in \eqref{eq:parameter-volume-functional} has a
	critical point when restricted to
	\(\operatorname{int}\Delta_\alpha\).
\end{proposition}

\begin{proof}

We apply Lemma~\ref{lem:relative-critical-point} with
\(\Delta=\Delta_\alpha\) and \(F=W\).
Lemma~\ref{lem:W-lower-bound} shows that \(W\) is bounded below,
Lemma~\ref{lem:W-blowup} gives condition~\textup{(i)}, and
Lemma~\ref{lem:W-regularity} gives condition~\textup{(ii)}.
Condition~\textup{(iii)} is vacuous if \(S=\varnothing\). 

Assume
that \(S\neq\varnothing\), and fix \(t^0\in S\).  For \(i\notin J(t^0)\), the tangent vector to \eqref{eq:exit-curve} at $s=0$ is
\begin{equation}\label{eq:exit-tangent-vector}
	v_i=e_i-\alpha_i t^0,
\end{equation}
where \(e_i\) is the \(i\)-th coordinate vector. 
Since the curve \eqref{eq:exit-curve} lies in
\(\Delta_\alpha\),  \(v_i\) belongs to the
tangent space of \(\operatorname{aff}\Delta_\alpha\). Let \(\widetilde W\) be the local \(C^1\) extension supplied by
Lemma~\ref{lem:W-regularity}. Since \(t'(0+)=v_i\) and
\(\widetilde W=W\) along \(t(s)\) for all sufficiently small
\(s\geq0\),  \eqref{eq:exit-derivative} gives

\begin{equation}\label{DW}
	DW(t^0)[v_i]
	=
	D\widetilde W(t^0)[v_i]
	=
	\left.\frac{d}{ds}\right|_{s=0+}W(t(s))
	=
	-\frac np\alpha_iW(t^0)
	>
	0.
\end{equation}
Set
\begin{equation}\label{eq:outward-descent-vector}
	Y_{t^0}
	=
	-\sum_{i\notin J(t^0)}v_i.
\end{equation}
For \(r\notin J(t^0)\), the facet
\(
G_r=\{t\in\Delta_\alpha:t_r=0\}
\)
contains \(t^0\). Since \(t_r^0=0\),
\[
(Y_{t^0})_r
=
-\sum_{i\notin J(t^0)}
\bigl(\delta_{ir}-\alpha_i t_r^0\bigr)
=
-1.
\]
Taking \(\ell_{G_r}(t)=t_r\) as a defining affine function for
\(G_r\) in Lemma~\ref{lem:relative-critical-point}, we have
\(D\ell_{G_r}[Y_{t^0}]=(Y_{t^0})_r=-1<0\).
Since \(r\notin J(t^0)\) was arbitrary, \(Y_{t^0}\) points strictly
outside every facet of \(\Delta_\alpha\) containing \(t^0\).

Moreover, by \eqref{DW} and
\eqref{eq:outward-descent-vector},
\[
DW(t^0)[Y_{t^0}]
=
-\sum_{i\notin J(t^0)}DW(t^0)[v_i]
=
\frac np W(t^0)\sum_{i\notin J(t^0)}\alpha_i
<0.
\]
Thus, Lemma~\ref{lem:relative-critical-point} applies. Lemma~\ref{lem:relative-homology-general} supplies the required nonzero relative homology.
\end{proof}

\begin{proof}[Proof of Theorem~\ref{thm:intro-homological}]
	By Proposition~\ref{prop:interior-critical-point}, there exists
	\(t^*\in\operatorname{int}\Delta_\alpha\) which is a critical point
	of \(W|_{\operatorname{int}\Delta_\alpha}\). Thus
 the Lagrange
	multiplier theorem gives a constant \(\lambda\in\R\) such that
	\begin{equation}\label{eq:general-Lagrange}
		\frac{\partial W}{\partial t_i}(t^*)
		=
		\lambda\alpha_i,
		\qquad
		1\leq i\leq N.
	\end{equation}
Euler's identity for the homogeneous function in \eqref{eq:inverse-volume-homogeneity} gives
\begin{equation}\label{eq:general-Euler}
 \sum_{i=1}^Nt_i^*
 \frac{\partial W}{\partial t_i}(t^*)
 =\frac np W(t^*).
\end{equation}
Combining \eqref{eq:weighted-simplex}, \eqref{eq:general-Lagrange}, and \eqref{eq:general-Euler} yields
\begin{equation}\label{eq:general-lambda}
 \lambda=\frac np W(t^*).
\end{equation}
Let $P^*=P(t^*)$, and let $m_i^*$ be its $L_p$ surface area mass at $u_i$. Equations \eqref{eq:inverse-volume-derivative}, \eqref{eq:general-Lagrange}, and \eqref{eq:general-lambda} imply
\begin{equation}\label{eq:general-masses}
 m_i^*=nW(t^*)\alpha_i,
 \qquad 1\leq i\leq N.
\end{equation}
Set
\[
 h^*=\bigl((t_1^*)^{1/p},\ldots,(t_N^*)^{1/p}\bigr).
\]
Every $m_i^*$ is positive by \eqref{eq:general-masses} and \(t^*\in\operatorname{int}\Delta_\alpha\), and \eqref{eq:one-sided-mass} gives
\[
 m_i^*=(h_i^*)^{1-p}a_i(h^*).
\]
Since $h_i^*>0$, it follows that $a_i(h^*)>0$ for every $i$. Thus, each prescribed direction is a genuine facet normal. 
Finally, set
\[
 c=\bigl(nW(t^*)\bigr)^{-1/(n-p)}.
\]

Since
\[
S_p(P^*,\cdot)
=
\sum_{i=1}^N m_i^*\delta_{u_i}
=
nW(t^*)\sum_{i=1}^N\alpha_i\delta_{u_i}
\]
by \eqref{eq:general-masses}, 
\eqref{eq:Lp-scaling} gives
\[
S_p(cP^*,\cdot)
=
\sum_{i=1}^N\alpha_i\delta_{u_i}.
\]

\end{proof}

\begin{remark}
	The homological condition in
	Theorem~\ref{thm:intro-homological} depends only on the direction
	configuration \(U\) and guarantees solvability for every positive
	choice of masses; it is not claimed to be necessary for a fixed mass
	vector. Together with Lemma~\ref{lem:no-essential-subspaces}, the
	theorem recovers the discrete result in
	\cite[Theorem~1.1]{ZhuNegative}. Indeed, if \(U\) has no essential
	subspaces and is not contained in a closed hemisphere, then
	\eqref{eq:zhu-deletion-homology} holds, so
	Theorem~\ref{thm:intro-homological} applies. 
\end{remark}

\section{Antipodal and even data}\label{sec:antipodal-even}

We now derive some consequences of
Theorem~\ref{thm:intro-homological}, beginning with a complete
characterization for antipodal supports.

\begin{proof}[Proof of Theorem~\ref{thm:intro-antipodal}]
	Suppose that \eqref{eq:intro-antipodal-target} holds.
	Lemma~\ref{lem:discrete-solution-polytope} shows that the solution is a
	polytope whose facet normals are precisely the points \(\pm u_i\), and
	these normals positively span \(\R^n\). Since the support is antipodal,
	this is equivalent to \eqref{eq:intro-spanning-condition}.
	
	Conversely, assume \eqref{eq:intro-spanning-condition}, and set
	\(Z=(u_1,-u_1,\ldots,u_N,-u_N)\). The list \(Z\) is antipodal.
	Since \(Z=-Z\), we have
	\(\pos Z=\spanop Z\). Hence
	\(\pos Z=\spanop\{u_1,\ldots,u_N\}=\R^n\) by
	\eqref{eq:intro-spanning-condition}, and \(Z\) is ample by
	\eqref{eq:positive-span-separation}.
	 Hence
	Lemma~\ref{thm:closed-hemisphere-topology} gives
	\[
	\widetilde H_{2n-2}\bigl(\Ncal_+(Z);\mathbb Z_2\bigr)\neq0.
	\]
	Since \(Z\) has \(2N\) points and
	\(\Ncal_+(Z)^*=\Dcal_+(Z)\),
	\eqref{eq:Alexander-nonvanishing} yields
	\[
	\widetilde H_{2N-2n-1}
	\bigl(\Dcal_+(Z);\mathbb Z_2\bigr)\neq0,
	\]
	with the convention \eqref{eq:minus-one-homology} when \(N=n\).
	Theorem~\ref{thm:intro-homological}, applied to the directions in
	\(Z\) and the masses \(\alpha_i^+\) and \(\alpha_i^-\), therefore
	gives the required polytope, with outer unit normals exactly the
	points of \(Z\).
\end{proof}

For even discrete data, Theorem~\ref{thm:intro-antipodal} already
guarantees the existence of a solution, but it remains to show that
a solution can be chosen to be origin-symmetric. Write
\begin{equation}\label{eq:even-data}
	\mu
	=
	\sum_{i=1}^N\alpha_i(\delta_{u_i}+\delta_{-u_i}),
	\qquad
	\alpha_i>0,
\end{equation}
where \(u_1,\ldots,u_N\in S^{n-1}\) span \(\R^n\) and satisfy
\(u_i\neq\pm u_j\) whenever \(i\neq j\), and set \(E=\{1,\ldots,N\}\). For
\(t\in(0,\infty)^N\), define
\begin{equation}\label{eq:symmetric-polytope}
	P_e(t)
	=
	\bigcap_{i=1}^N
	\{x\in\R^n:|\langle x,u_i\rangle|\leq t_i^{1/p}\},
	\qquad
	W_e(t)=V(P_e(t)).
\end{equation}
The spanning assumption implies that \(P_e(t)\) is a bounded
origin-symmetric polytope containing the origin in its interior. Let
\[
m_i^e(t)
=
S_p(P_e(t),\{u_i\})
=
S_p(P_e(t),\{-u_i\}).
\]

\begin{lemma}\label{lem:symmetric-volume}
	The function \(W_e\) belongs to \(C^1((0,\infty)^N)\) and satisfies
	\begin{equation}\label{eq:symmetric-inverse-derivative}
		\frac{\partial W_e}{\partial t_i}(t)
		=
		\frac2p\,m_i^e(t),
		\qquad
		W_e(ct)=c^{n/p}W_e(t)
		\quad(c>0).
	\end{equation}
\end{lemma}

\begin{proof}

	Apply Corollary~\ref{cor:inverse-volume-derivative} to the signed
	configuration
	\((u_1,-u_1,\ldots,u_N,-u_N)\), and restrict the corresponding
	\(2N\) variables to
	\((t_1,t_1,\ldots,t_N,t_N)\). This restriction shows that
	\(W_e\in C^1((0,\infty)^N)\). When \(t_i\) varies, the two variables
	corresponding to \(u_i\) and \(-u_i\) vary simultaneously, so the
	chain rule gives the sum of the two corresponding partial
	derivatives. Since \(P_e(t)\) is origin-symmetric, its \(L_p\)
	surface area masses in the directions \(u_i\) and \(-u_i\) are
	equal. This proves the first formula in
	\eqref{eq:symmetric-inverse-derivative}. 
	The homogeneity follows from 
	\(P_e(ct)=c^{1/p}P_e(t)\).

\end{proof}

Define the paired weighted simplex
\begin{equation}\label{eq:symmetric-weighted-simplex}
	\Delta_\alpha^e
	=
	\left\{t\in[0,\infty)^N:\sum_{i=1}^N\alpha_it_i=1\right\},
	\qquad
	J(t)=\{i:t_i>0\}.
\end{equation}
Let
\begin{equation}\label{eq:symmetric-nonspanning-complex}
	\Ncal_e(U)
	=
	\left\{I\subset E:\spanop\{u_i:i\in I\}\neq\R^n\right\}.
\end{equation}
For \(I\subset E\), put
\(\Delta_I^e=\{t\in\Delta_\alpha^e:t_j=0\text{ for }j\notin I\}\),
and set
\begin{equation}\label{eq:symmetric-AXS}
	A_e=\bigcup_{I\in\Ncal_e(U)}\Delta_I^e,
	\qquad
	X_e=\Delta_\alpha^e\setminus A_e,
	\qquad
	S_e=\partial\Delta_\alpha^e\setminus A_e.
\end{equation}
Since the full family spans \(\R^n\), we have
\(E\notin\Ncal_e(U)\); hence \(A_e\subset\partial\Delta_\alpha^e\) and
\(X_e=\operatorname{int}\Delta_\alpha^e\cup S_e\). Moreover,
\[
X_e
=
\left\{
t\in\Delta_\alpha^e:
\spanop\{u_i:i\in J(t)\}=\R^n
\right\}.
\]
For \(t\in\Delta_\alpha^e\), extend
\eqref{eq:symmetric-polytope} by omitting the zero coordinates:
\[
P_e(t)
=
\bigcap_{i\in J(t)}
\left\{
x\in\R^n:
|\langle x,u_i\rangle|\leq t_i^{1/p}
\right\}.
\]
Similarly to Subsection~\ref{subsec:weighted-parameter-simplex}, \(X_e\) is precisely the set of parameters for which
\(P_e(t)\) is bounded; equivalently, \(P_e(t)\) is unbounded
exactly when \(t\in A_e\). For \(t\in X_e\), set
\begin{equation}\label{eq:symmetric-boundary-functional}
	W_e(t)=V(P_e(t)).
\end{equation}

\begin{lemma}\label{lem:symmetric-critical-point}
	The volume functional \(W_e\) defined in
	\eqref{eq:symmetric-boundary-functional} has a critical point when
	restricted to \(\operatorname{int}\Delta_\alpha^e\).
\end{lemma}

\begin{proof}
Apply the one-sided construction of
Subsection~\ref{subsec:weighted-parameter-simplex} to the signed
configuration
\((u_1,-u_1,\ldots,u_N,-u_N)\), assigning the weight
\(\alpha_i/2\) to each of \(u_i\) and \(-u_i\). Set the two
parameters corresponding to \(u_i\) and \(-u_i\) equal to \(t_i\).
Then the one-sided normalization reduces to
\(\sum_i\alpha_i t_i=1\), and the associated one-sided polytope is
exactly \(P_e(t)\). Moreover,
	\[
	\pos\{\pm u_i:i\in J(t)\}
	=
	\spanop\{u_i:i\in J(t)\}.
	\]
	Under this restriction, the one-sided definitions in
	\eqref{eq:AXS} and \eqref{eq:parameter-volume-functional} reduce
	exactly to \eqref{eq:symmetric-AXS} and
	\eqref{eq:symmetric-boundary-functional}. Hence
	Lemma~\ref{lem:W-lower-bound} shows that \(W_e\) is bounded below,
	Lemma~\ref{lem:W-blowup} gives condition~\textup{(i)}, and
	Lemma~\ref{lem:W-regularity} gives condition~\textup{(ii)} of
	Lemma~\ref{lem:relative-critical-point}.

	We verify condition~\textup{(iii)} by an argument similar to those
	in the proofs of Lemma~\ref{lem:exit-derivative} and
	Proposition~\ref{prop:interior-critical-point}.
	Condition~\textup{(iii)} is vacuous if \(S_e=\varnothing\). Suppose
	that \(S_e\neq\varnothing\), fix \(t^0\in S_e\), and let
	\(i\notin J(t^0)\). Along the curve \eqref{eq:exit-curve}, we have
	\(t(s)\in X_e\) for \(s>0\), and
	\[
	P_e(t(s))
	=
	(1-\alpha_i s)^{1/p}P_e(t^0)
	\cap
	\{x:|\langle x,u_i\rangle|\leq s^{1/p}\}.
	\]
	Since \(P_e(t^0)\) is bounded and \(s^{1/p}\to+\infty\),  for all sufficiently small \(s>0\), we have
	\[
	W_e(t(s))=(1-\alpha_i s)^{n/p}W_e(t^0).
	\]
For each \(i\notin J(t^0)\), set
\(v_i=t'(0+)=e_i-\alpha_i t^0\). Applying the chain rule to the
local \(C^1\) extension of \(W_e\) and differentiating the preceding
formula, we obtain
\[
DW_e(t^0)[v_i]
=
\left.\frac{d}{ds}\right|_{s=0+}W_e(t(s))
=
-\frac np\alpha_iW_e(t^0)
>
0.
\] The simplex calculation in the proof of
	Proposition~\ref{prop:interior-critical-point} therefore applies
	unchanged: the vector
	\(Y_{t^0}^e=-\sum_{i\notin J(t^0)}v_i\) belongs to the translation
	space of \(\operatorname{aff}\Delta_\alpha^e\), points strictly
	outside every facet containing \(t^0\), and satisfies
	\[
	DW_e(t^0)[Y_{t^0}^e]
	=
	\frac np W_e(t^0)
	\sum_{i\notin J(t^0)}\alpha_i
	<0,
	\]
	where the strict inequality follows from \(p<0\) and
	\(t^0\in\partial\Delta_\alpha^e\). Thus condition~\textup{(iii)} holds.
	
	By the definition \eqref{eq:combinatorial-Alexander-dual}, \eqref{eq:symmetric-nonspanning-complex} and \eqref{eq:even-deletion-complex},
	\[
	J\in\Ncal_e(U)^*
	\iff
	E\setminus J\notin\Ncal_e(U)
	\iff
	\spanop\{u_i:i\in E\setminus J\}=\R^n
	\iff
	J\in\Dcal_e(U).
	\]
	Hence \(\Ncal_e(U)^*=\Dcal_e(U)\), and
	Lemma~\ref{lem:even-deletion-homology} gives
	\[
	\widetilde H_{N-n-1}
	\bigl(\Ncal_e(U)^*;\mathbb Z_2\bigr)\neq0.
	\]
	As in the proof of Lemma~\ref{lem:relative-homology-general},
	\(X_e\) is contractible and
	\(S_e\simeq\lvert\Dcal_e(U)\rvert\).
	The relative homology argument in that proof, with
	\(k=N-n-1\), therefore gives
	\[
	H_{N-n}(X_e,S_e;\mathbb Z_2)\neq0.
	\]
	Lemma~\ref{lem:relative-critical-point} now applies with
	\(q=N-n\), proving the lemma.
\end{proof}

\begin{proof}[Proof of Corollary~\ref{cor:intro-even}]
	Necessity follows from Lemma~\ref{lem:discrete-solution-polytope}:
	for an antipodal support, positive spanning is equivalent to linear
	spanning, which in turn is equivalent to not being contained in a
	great subsphere.
	
	Conversely, write \(\mu\) as in \eqref{eq:even-data}. The hypothesis
	implies that \(u_1,\ldots,u_N\) span \(\R^n\). 
By Lemma~\ref{lem:symmetric-critical-point}, there exists
\(t^*\in\operatorname{int}\Delta_\alpha^e\) which is a critical
point of \(W_e|_{\operatorname{int}\Delta_\alpha^e}\).
Thus the Lagrange multiplier theorem gives a constant \(\lambda\in\R\)
such that
\[
\frac{\partial W_e}{\partial t_i}(t^*)
=
\lambda\alpha_i,
\qquad 1\leq i\leq N.
\]
	Euler's identity and \(\sum_i\alpha_it_i^*=1\) give
	\[
	\lambda
	=
	\sum_{i=1}^Nt_i^*\frac{\partial W_e}{\partial t_i}(t^*)
	=
	\frac np W_e(t^*).
	\]
	Together with \eqref{eq:symmetric-inverse-derivative}, this yields
	\[
	m_i^e(t^*)
	=
	\frac{nW_e(t^*)}{2}\alpha_i>0.
	\]
	Consequently, each \(u_i\) and \(-u_i\) is the outer unit normal of
	a facet of \(P_e(t^*)\), and
	\[
	S_p(P_e(t^*),\cdot)
	=
	\frac{nW_e(t^*)}{2}\mu.
	\]
	For
	\[
	c=\left(\frac{2}{nW_e(t^*)}\right)^{1/(n-p)},
	\]
	the scaling relation \eqref{eq:Lp-scaling} gives
	\(S_p(cP_e(t^*),\cdot)=\mu\). The polytope \(cP_e(t^*)\) is
	origin-symmetric and has outer unit normals exactly the points of
	\(\supp\mu\).
\end{proof}

\medskip
\noindent\emph{A planar counterexample.}
The following proposition shows that positive spanning alone is not
sufficient for general discrete data.

\begin{proposition}\label{prop:planar-counterexample}
	In \(\R^2\), let
	\begin{equation}\label{eq:counterexample-directions}
		u_1=(1,0),\qquad
		u_2=\left(-\frac12,\frac{\sqrt3}{2}\right),\qquad
		u_3=(-1,0),\qquad
		u_4=\left(-\frac12,-\frac{\sqrt3}{2}\right).
	\end{equation}
	Then
	\[
	\pos\{u_1,u_2,u_3,u_4\}=\R^2,
	\]
	but there is no \(P\in\mathcal K_o^2\) such that
	\begin{equation}\label{eq:counterexample-target}
		S_{-2}(P,\cdot)
		=
		\sum_{i=1}^4\delta_{u_i}.
	\end{equation}
\end{proposition}

\begin{proof}
Clearly, the directions in \eqref{eq:counterexample-directions}
positively span \(\R^2\).
	Suppose that a solution \(P\) of
	\eqref{eq:counterexample-target} exists. By
	Lemma~\ref{lem:discrete-solution-polytope}, \(P\) is a
	quadrilateral whose outer unit normals are precisely
	\(u_1,u_2,u_3,u_4\). Write \(a=h_P(u_1)\), \(b=h_P(u_2)\), \(c=h_P(u_3)\), and
	\(d=h_P(u_4)\).
	The four edges with outer normals \(u_1,u_2,u_3,u_4\) lie,
	respectively, on the lines
	\[
	x=a,\qquad
	-x+\sqrt3\,y=2b,\qquad
	x=-c,\qquad
	-x-\sqrt3\,y=2d.
	\]
	Each vertex lies on two adjacent edges, so its coordinates satisfy
	both corresponding equations. Solving the corresponding pairs of equations gives the vertex coordinates
	\[
	\left(a,\frac{a+2b}{\sqrt3}\right),
	\quad
	\left(-c,\frac{2b-c}{\sqrt3}\right),
	\quad
	\left(-c,\frac{c-2d}{\sqrt3}\right),
	\quad
	\left(a,-\frac{a+2d}{\sqrt3}\right).
	\]
	Let \(\ell_i\) denote the length of the edge with outer normal
	\(u_i\). Taking the distances between consecutive vertices gives
	\begin{equation}\label{eq:counterexample-edge-lengths}
		\ell_1=\frac{2}{\sqrt3}(a+b+d),
		\qquad
		\ell_2=\frac{2}{\sqrt3}(a+c),
		\qquad
		\ell_3=\frac{2}{\sqrt3}(b+d-c),
		\qquad
		\ell_4=\frac{2}{\sqrt3}(a+c).
	\end{equation}
	In particular, \(\ell_3>0\) implies \(b+d-c>0\).
	
	The \(L_{-2}\) mass of an edge with support number
	\(h\) and length \(\ell\) is \(h^3\ell\). Since the four masses in
	\eqref{eq:counterexample-target} are equal, 
	\eqref{eq:counterexample-edge-lengths} gives
	\begin{equation}\label{eq:counterexample-mass-system}
		a^3(a+b+d)
		=
		b^3(a+c)
		=
		c^3(b+d-c)
		=
		d^3(a+c).
	\end{equation}
	The second and fourth terms give \(b=d\). Since all four
	expressions in \eqref{eq:counterexample-mass-system} are
	homogeneous of degree four, we may
	therefore assume that \(b=d=1\). Then
	\begin{equation}\label{eq:counterexample-reduced-system}
		a^3(a+2)
		=
		a+c
		=
		c^3(2-c),
		\qquad
		0<c<2.
	\end{equation}
	
	Since \(c^3(2-c)=a+c>c\), we have
	\(c^2(2-c)>1\). For \(0<c\leq1\),
	\[
	c^2(2-c)
	\leq
	c(2-c)
	=
	1-(1-c)^2
	\leq1.
	\]
	Hence \(c>1\). On the other hand, the function
	\(q(c)=c^3(2-c)\) attains its maximum on \((0,2)\) at
	\(c=3/2\), and therefore
	\[
	c^3(2-c)\leq\frac{27}{16}.
	\]
	It follows from \eqref{eq:counterexample-reduced-system} and
	\(c>1\) that \(a<11/16\). Consequently,
	\begin{equation}\label{eq:counterexample-left-bound}
		a^3(a+2)
		<
		\left(\frac{11}{16}\right)^3
		\frac{43}{16}
		<
		1.
	\end{equation}
	However, \eqref{eq:counterexample-reduced-system} also gives
	\(a^3(a+2)=a+c>1\), contradicting
	\eqref{eq:counterexample-left-bound}.
\end{proof}

\begin{remark}
	For the direction set
	\(U=\{u_1,u_2,u_3,u_4\}\) in
	Proposition~\ref{prop:planar-counterexample}, a direct check gives
	\[
	\Dcal_+(U)=\{\varnothing,\{3\}\}.
	\]
	Thus \(\lvert\Dcal_+(U)\rvert\) is a single point, and its reduced
	homology vanishes in every degree. The homological hypothesis of
	Theorem~\ref{thm:intro-homological} therefore does not apply.
	
Nevertheless, the quadrilateral with support numbers
\(a=b=c=d=1\) has \(L_{-2}\) mass vector
\(\frac{2}{\sqrt3}(3,2,1,2)\), by
\eqref{eq:counterexample-edge-lengths}.
For the same directions \(U\), this mass vector is realizable,
whereas \((1,1,1,1)\) is not.
Thus, when the homological condition fails, solvability can
depend on the prescribed masses.
\end{remark}

\par\medskip
\noindent\textbf{Declaration of AI use.}
AI was used to assist with literature searches, particularly those used in Lemmas~\ref{thm:closed-hemisphere-topology}
and~\ref{lem:even-deletion-homology}, and with English-language
editing.
The author takes full responsibility for the content.

\end{document}